\documentclass{amsart}
\usepackage{xcolor}
\usepackage{amsmath,amssymb,amsthm,amsfonts,amscd,mathrsfs,mathtools,enumerate}
\usepackage{enumitem}
\usepackage{tikz,graphicx}
\usepackage[all]{xy}
\usepackage[hyphens]{url}
\usepackage{hyperref}
\usepackage{caption}
\usepackage{tikz-cd}
\usepackage{booktabs}
\usepackage{graphicx}
\usepackage{comment}
\usetikzlibrary{decorations.markings}

\theoremstyle{plain}
    \newtheorem{thm}{Theorem}[section]
    \newtheorem{lem}[thm]   {Lemma}
    
    \newtheorem{prop}[thm]  {Proposition}
    
\theoremstyle{definition}
    \newtheorem{defn}[thm]  {Definition}

    \newtheorem{ex}[thm]{Example}
    \newtheorem{rem}[thm]{Remark}

\def\max{\mathrm{max}}

\def\dim{\mathrm{dim}}

\def\spspan{\mathrm{s\, pspan}}
\def\pspan{\mathrm{pspan}}
\def\span{\mathrm{span}}
\def\sspan{\mathrm{s\, span}}

\newcommand{\be}{\begin{enumerate}}
\newcommand{\ee}{\end{enumerate}}
\newcommand{\R}{\mathbb{R}}
\newcommand{\Z}{\mathbb{Z}}
\newcommand{\C}{\mathbb{C}}

\newcommand{\la}{\langle}
\newcommand{\ra}{\rangle}

\usetikzlibrary{calc}

\usetikzlibrary{decorations.pathmorphing}

\tikzset{curve/.style={settings={#1},to path={(\tikztostart)
    .. controls ($(\tikztostart)!\pv{pos}!(\tikztotarget)!\pv{height}!270:(\tikztotarget)$)
    and ($(\tikztostart)!1-\pv{pos}!(\tikztotarget)!\pv{height}!270:(\tikztotarget)$)
    .. (\tikztotarget)\tikztonodes}},
    settings/.code={\tikzset{quiver/.cd,#1}
        \def\pv##1{\pgfkeysvalueof{/tikz/quiver/##1}}},
    quiver/.cd,pos/.initial=0.35,height/.initial=0}

\tikzset{tail reversed/.code={\pgfsetarrowsstart{tikzcd to}}}
\tikzset{2tail/.code={\pgfsetarrowsstart{Implies[reversed]}}}
\tikzset{2tail reversed/.code={\pgfsetarrowsstart{Implies}}}

\tikzset{no body/.style={/tikz/dash pattern=on 0 off 1mm}}

\begin{document}

\title[Projective span]{Projective span of Wall manifolds}

\author{Mark Grant}

\author{Baylee Schutte}

\address{Institute of Mathematics,
Fraser Noble Building,
University of Aberdeen,
Aberdeen AB24 3UE,
UK}

\email{mark.grant@abdn.ac.uk}
\email{b.schutte.21@abdn.ac.uk}

\date{\today}

\keywords{span, projective span, vector fields, line fields}
\subjclass[2010]{57R25 (Primary); 55S40, 57R19, 15A66 (Secondary).}

\begin{abstract} The projective span of a smooth manifold is defined to be the maximal number of linearly independent tangent line fields. We initiate a study of projective span, highlighting its relationship with the span, a more classical invariant. We calculate the projective span for all Wall manifolds, which are certain mapping tori of Dold manifolds.
\end{abstract}


\maketitle
\section{Introduction}\label{sec:intro}

Let $M$ be a closed smooth manifold of dimension $m$. A \emph{line field} on $M$ is a smooth section $\ell:M\to PTM$ of the projectivized tangent bundle of $M$. A finite set $\{\ell_1,\ldots , \ell_k\}$ of line fields on $M$ is linearly independent if for each $x\in M$ the lines $\ell_1(x),\ldots, \ell_k(x)$ generate a $k$--dimensional subspace of $TM_x$. 

\begin{defn}
The \emph{projective span} of $M$, denoted $\pspan(M)$, is defined to be the maximal $k$ such that $M$ admits $k$ linearly independent line fields.
\end{defn}

Calculation of the projective span of a given manifold or family of manifolds will be called the \emph{line field problem}. Recall that the \emph{span} of $M$, denoted $\span(M)$, is the maximal $k$ such that $M$ admits $k$ linearly independent vector fields; its calculation is known as the \emph{vector field problem}. Since linearly independent vector fields span linearly independent line fields, $\pspan(M)\geq \span(M)$ always. The following basic example (noted already in \cite{MS}) shows that inequality can be strict.

\begin{ex} The Klein bottle $K$ has $\pspan(K)=2>1=\span(K)$.

The standard construction of $K$ as a quotient of the plane indicates $2$ linearly independent line fields, giving $\pspan(K)=2$. One of these line fields lifts to a vector field, giving $\span(K)\ge1$. But $K$ is not parallelizable since it is non-orientable, so $\span(K)<2$. \qed
\end{ex}

Let $\chi(M)$ denote the Euler characteristic of $M$. It is a classical result due to H.\ Hopf that $\span(M)>0$ if and only if $\chi(M)=0$. It is also well-known that $\pspan(M)>0$ if and only if $\chi(M)=0$ \cite{Sam,Markus,CG}. 

At the other extreme, an $m$--dimensional manifold $M$ is called \emph{parallelizable} if $\span(M)=m$, and is called \emph{line element parallelizable} if $\pspan(M)=m$. Massey and Szczarba \cite{MS} give an example of an orientable $4$--manifold $M$ which is line element parallelizable but not parallelizable. Their example has $\pspan(M)=4>2=\span(M)$. Auslander and Szczarba \cite{AS} showed that all compact solvmanifolds are line element parallelizable. Further results on line element parellelizability were obtained by Alliston \cite{Alliston} and by W.\ Iberkleid \cite{Iberkleid}, who showed that $M$ is unoriented cobordant to a line element parellelizable manifold if and only if $\chi(M)$ is even.

In this paper, we solve the line field problem for Wall manifolds. The Wall manifolds $Q(m,n)$ are certain mapping tori of Dold manifolds $P(m,n)$, used by Wall in his determination of the oriented cobordism ring \cite{Wall}. The vector field problem has been studied by Novotn\'y \cite{Novotny} for Dold manifolds, and by Khare \cite{Khare} for Wall manifolds; however in most cases the exact value of $\span(P(m,n))$ or $\span(Q(m,n))$ remains unknown.

\begin{thm}\label{thm:main} 
Given \(m \geq 1\) and \(n \geq 0\), the projective span of the Wall manifolds satisfies
\[
\pspan(Q(m,n)) = 2\nu(n+1) + m + 1,
\]
where \(\nu(k)\) denotes the largest exponent such that the corresponding \(2\)--power divides the integer \(k\).
\end{thm}

\begin{rem}
Using the span calculations of Khare \cite{Khare}, we see that 
\[
1 = \span(Q(m,n)) < \pspan(Q(m,n)) = m+1
\]
for \(m \text{ and } n\) both even, showing that the difference between span and projective span can be arbitrarily large. 
\end{rem}

The proof of Theorem \ref{thm:main} has two main ingredients. The first is the calculation of the stable span of $\C P^n$ due to Steer \cite{Steer}, who used complex Clifford modules to produce vector fields. The second is the observation that if $M$ is a free $G$-manifold with $G$ a compact Lie group, then vector fields on $M$ that are \emph{$G$-quasi-invariant}, namely, $G$-invariant up to a sign, induce line fields on the quotient $M/G$.
 
Our motivation for studying projective span comes from two sources. Firstly, the vector field problem has a rich history and has inspired a great deal of deep mathematics. We mention here Adams' solution to the vector field problem on spheres \cite{Adams}; the index theoretical approach of Atiyah  and Dupont \cite{AD}; the approach using obstruction theory and higher cohomology operations pioneered by Emery Thomas and others, as surveyed for example in \cite{Thomas}; and Koschorke's approach \cite{Koschorke} using singularity theory and normal bordism. It is natural to hope that the study of projective span will similarly lead to new theoretical advances.

Our second motivation comes from soft-matter physics. There line fields have been used to model nematic liquid crystals, that is, ordered media consisting of rod-shaped molecules with no preferred axial direction which tend to align with each other in regions of space. Similarly, tuples of linearly independent line fields can be used to model ``biaxial nematics'', which is a catch-all term used in physics for ordered media with more than one axis of alignment. Synthesis of the  biaxial nematic phase is something of a holy-grail in experimental soft-matter physics and chemistry, and is expected to have important applications to display technology.  A proper mathematical understanding of the toplogical obstructions to the existence of independent line fields may lead to a deeper understanding of the singularities or defects in such ordered media, and therefore may have practical applications.      

We thank Cihan Bahran and Ran Levi for valuable input regarding this work. We especially thank Markus Upmeier for guiding us through the literature on Clifford algebras and for helpful comments on an earlier draft.

\section{Introduction to the Projective Span}\label{sec:props}

\subsection{Basic properties of projective span}\label{subsec:props}
In this section we collect some basic properties of the projective span, highlighting its relationship to span. We also outline the known tools for computing projective span and give several interesting examples. 

We enumerate a few obvious properties of the projective span. 
\begin{prop}\label{prop:props}
If \(M\) is a smooth closed connected \(m\)--manifold then
\begin{enumerate}[label=\emph{(\roman*)}]
    \item \(\pspan(M) \geq k\) if and only if \(TM \cong L_1 \oplus \cdots \oplus L_k \oplus \xi^{m-k}\), where the \(L_i\), for \(1 \leq i \leq k\), are possibly nontrivial line bundles over \(M\), and \(\xi^{m-k}\) is the rank \((m-k)\) complementary bundle. 
    \item \(\pspan(M) \neq m-1.\)
    \item \(\pspan(M) \geq \span(M).\)
    \item If \(M\) is simply connected then \(\pspan(M) = \span(M).\)
    \item If \(\widetilde{M}\to M\) is a covering space, then \(\pspan(\widetilde{M}) \geq \pspan(M).\)
\end{enumerate}
\end{prop}
\begin{proof}[Proof of \emph{(i)}]
This follows since a line field \(\ell \colon M \to PTM\) is a one dimensional subbundle \(\ell(M) \subseteq TM\) of the tangent bundle of \(M\). \phantom{\qedhere}
\end{proof}

\begin{proof}[Proof of \emph{(ii)}] Assume for a contradiction that \(\pspan(M) = m-1\). Then by (i), \(TM \cong L_1 \oplus \cdots \oplus L_{m-1} \oplus \xi^1,\) but \(\xi^1 \) is one dimensional, whence \(\pspan(M) = m.\)\phantom{\qedhere}
\end{proof}

\begin{proof}[Proof of \emph{(iii)}] A non-vanishing vector field \(v \colon M \to TM\) gives rise to a line field \(\ell \colon M \to PTM\) by taking 
\[
\ell(x) \coloneqq \langle v(x) \rangle \subseteq TM_x
\]
to be the line spanned by \(v(x)\) at each point.\phantom{\qedhere}
\end{proof}

\begin{proof}[Proof of \emph{(iv)}] Suppose \(\pspan(M) \geq k\). Thus \(TM \cong L_1 \oplus \cdots \oplus L_k \oplus \xi^{m-k}\). The assumption that \(M\) is simply connected forces each \(L_i\), for \(i = 1, \dots, k\), to be trivial, since line bundles on $M$ are classified by the first cohomology group $H^1(M;\Z_2) \cong {\rm{Hom}}(\pi_1(M), \Z_2) = 0$. Hence \(TM \cong \varepsilon^k \oplus \xi^{m-k}\) and \(\span(M) \geq \pspan(M).\) Finally, equality holds by (iii). \phantom{\qedhere}
\end{proof}

\begin{proof}[Proof of \emph{(v)}] Suppose \(\pspan(M) \geq k\). Then \(TM = L_1 \oplus \cdots \oplus L_k \oplus \xi^{m-k}\). Then 
\[
T\widetilde{M} = p^*(TM) = p^*L_1 \oplus \cdots \oplus p^*L_k \oplus p^*\xi^{m-k},
\]
but each \(p^*L_i\), for \(1 \leq i \leq k\), is a line bundle and the result follows. 
\end{proof}

\begin{rem}
In the proof of Proposition \ref{prop:props}(iii), we see that every non-vanishing vector field gives rise to a line field; of course, not every line field arises from a vector field in this way. Indeed, a line field \(\ell\) lifts to a vector field if and only if its associated double cover \(p \colon S(\ell) \to M\) is trivial, where \(S(\ell) \subseteq S(TM)\) denotes the associated sphere bundle. 
\end{rem}

\subsection{Examples}\label{subsec:ex}

\begin{ex}[Projective span of real projective space]\label{ex:rpn}
Let \(n \geq 1\) be a natural number. Hurwitz and Radon independently computed that \(\span(S^{n-1}) \geq \rho(n) - 1\), where \(\rho(n)\) is the \emph{Hurwitz--Radon number.} Famously, Adams went on to show that \(\span(S^{n-1}) = \rho(n) -1\). Next, one can show that the \(\rho(n)-1\) vector fields on \(S^{n-1}\) descend to the quotient \(S^{n-1}/\Z_2 = \R P^{n-1}\), see \cite{Sankaran} for a modern proof. This implies that \(\span(\R P^{n-1}) = \span(S^{n-1}).\)

We claim that \(\pspan(\R P^{n-1}) = \span(S^{n-1})=\rho(n)-1.\) Indeed, we have that 
\begin{equation}\label{eq:pspanRPn}
\begin{split}
\span(\R P^{n-1})= \span(S^{n-1})&\overset{(a)}=\pspan(S^{n-1}) \\ &\overset{(b)}\geq \pspan(\R P^{n-1}) \overset{(c)}\geq \span(\R P^{n-1}).
\end{split}
\end{equation}
Therefore the inequalites of (\ref{eq:pspanRPn}) are actually equalities, and the claim follows. By way of explanation, (a) follows from Proposition \ref{prop:props}(iv); (b) follows from Proposition \ref{prop:props}(v); and (c) follows from Proposition \ref{prop:props}(iii).
\end{ex}

\begin{ex}[Estimates of the projective span of some flag manifolds]\label{ex:flag}
Suppose \(n_1,\dots, n_s\) are fixed positive integers such that \(n_1 + n_2 + \dots +n_s = n\). A \emph{\((n_1,\dots, n_s)\)--flag over \(\R\)} consists of a collection \(\sigma\) of mutually orthogonal subspaces \((\sigma_1,\dots, \sigma_s)\) in \(\R^n\) such that \({\rm{dim}}_{\R}(\sigma_i) = n_i\). The space of all such flags forms a compact smooth manifold which we denote by \(\R F(n_1,\dots,n_s)\).

Define \(\xi_i\) to be the real vector bundle whose fiber at the point \(\sigma  = (\sigma_1,\dots, \sigma_s)\in \R F(n_1,\dots, n_s)\) is the vector space \(\sigma_i.\)\par 

It is a theorem of K.\ Y.\ Lam \cite{Lam} that the tangent bundle of \(F=\R F(n_1,\dots, n_s)\) splits as in 
\[
TF~~ \cong \bigoplus_{1 \leq i < j \leq s} \overline{\xi}_i\otimes_\R \xi_j,
\]
where \(\overline{\xi}_i\) denotes the dual of \(\xi_i\). This immediately gives a lower bound for the projective span \[\pspan(\R F(\underbrace{1, 1, \dots, 1}_k,n-k)) \geq \binom{k}{2}\] for all $2\le k\le n$. We do not know if this bound is sharp. On the other hand, very little is known about the span of the manifolds \(\R F(1,\dots, 1, n-k)\), see \cite{KZ} for an overview and \cite{AI} for more complete information about the case \(\R F(1,1,n-2).\)
\end{ex}

We now introduce the Dold manifolds and Wall manifolds, and we collect various  facts about them which we will find use for later. 

\begin{ex}[Dold Manifolds]\label{ex:Dold}
Let $m,n\ge0$ be non-negative integers. The Dold manifold $P(m,n)$ is defined to be the quotient of $S^m\times \C P^n$ by the diagonal action of $\Z/2$, where $\Z/2$ acts freely on $S^m$ by the antipodal map and on $\C P^n$ by complex conjugation. Therefore it is a manifold of dimension $(m+2n)$ which fits into a fibre bundle
\[
\xymatrix{
\C P^n \ar[r] & P(m,n) \ar[r] & \R P^m.
}
\]
The manifolds $P(m,n)$ were shown by Dold to provide generators of the odd dimensional cobordism groups. Their mod $2$ cohomology rings are given by
\[
H^*(P(m,n);\Z/2)\cong \Z/2[c,d]/(c^{m+1},d^{n+1})
\]
where $|c|=1$ and $|d|=2$.

In \cite{Korbas}, Korba\v{s} solves the parallelizability problem for Dold manifolds. He finds that among all \(P(m,n)\), only \(P(1,0) \cong \R P^1\), \(P(3,0) \cong \R P^3\), and \(P(7,0) \cong \R P^7\) are parallelizable. 

Suppose \(m \geq 0\) is even. Then the Euler characteristic \(\chi(P(m,1)) = 2\neq 0\), whence \(\span(P(m,1)) = \pspan(P(m,1)) = 0.\)

On the other hand, if \(m \geq 0\) is odd, then Korba\v{s} \cite{Korbas} has calculated that 
\[
\span(P(m,1)) = 1 + \span(S^m),
\]
from which \(P(1,1), P(3,1)\), and \(P(7,1)\) are examples of line element parallelizable manifolds which are not parallelizable. 
\end{ex}

\begin{ex}[Wall manifolds]\label{ex:Wall}
Let \(m,n\geq0\) be non-negative integers. The Wall manifolds $Q(m,n)$ are certain mapping tori of Dold manifolds. The involution on $S^m$ given by negating the last coordinate of $\R^{m+1}$ commutes with the antipodal map and therefore induces an involution $\rho: P(m,n)\to P(m,n)$. The Wall manifold $Q(m,n)$ is then defined to be the mapping torus of this involution. 

Wall \cite{Wall} described the mod $2$ cohomology ring of $Q(m,n)$. It is given by
\[
H^*(Q(m,n);\Z/2)\cong \Z/2[\overline{x},\overline{c},\overline{d}]/(\overline{x}^2,\overline{c}^{m+1}-\overline{c}^m\overline{x},\overline{d}^{n+1})
\]
where $|\overline{x}|=|\overline{c}|=1$ and $|\overline{d}|=2$. Additionally, Wall computed the total Stiefel--Whitney class to be
\[
w\big(Q(m,n)\big) = (1+\overline{c}+\overline{x})(1+\overline{c})^{m-1}(1+\overline{c}+\overline{d})^{n+1}.
\]

\end{ex}

\subsection{Tools for computing projective span.}\label{subsec:tools}
We would like for this paper to lead to further research on projective span; accordingly, we collect in this section the few known theoretical tools for its computation. 

If $\span(M)\ge k$ then the tangent bundle splits as $TM\cong \xi^{m-k}\oplus\varepsilon^k$, where $\xi$ is some bundle of rank $m-k$ and $\varepsilon^k$ denotes a trivial bundle of rank $k$. Therefore the Stiefel--Whitney classes $w_i(M)$ vanish for $i>m-k$. Similarly, if $\pspan(M)\ge k$ then there exist line bundles $L_1,\ldots , L_k$ on $M$ such that 
\[
TM\cong \xi^{m-k}\oplus L_1\oplus \cdots \oplus L_k.
\]
This implies that the virtual Stiefel--Whitney classes $w_i(TM-L_1\oplus \cdots \oplus L_k)$ vanish for $i>m-k$. These classes are readily computed in terms of the Stiefel--Whitney classes of $M$ and the classes $x_i=w_1(L_i)\in H^1(M;\Z/2)$ for $i=1,\ldots , k$ which classify the given line bundles, and give obstructions to the given line bundles splitting off the tangent bundle. 

Mello \cite{Mello} has given sufficient conditions for a splitting $TM\cong \xi\oplus\eta$, where $\eta$ is an arbitrary rank $2$ vector bundle. Applying her results in the case when $\eta\cong L_1\oplus L_2$ is a sum of line bundles, one can identify cases in which $w_{m-1}(TM-L_1\oplus L_2)$ is the sole obstruction to $L_1$ and $L_2$ splitting off $TM$. If this obstruction vanishes and $w_{m-1}(M)$ does not vanish, then $\span(M)<2\leq\pspan(M)$. 

Additionally, Massey and Szczarba \cite[Theorem A]{MS} have developed necessary conditions for projective span to be at least \(k\). Before stating these conditions, we set notation.

We write \(w_i(M) \in H^i(M;\Z_2)\), where \(1 \leq i \leq m\), for the \(i\)--th Stiefel--Whitney class of the tangent bundle \(TM\) of \(M^m\), and \(p_i(M) \in H^{4i}(M;\Z)\), where \(1 \leq i \leq m/4,\) for the \(i\)--th Pontryagin class. Furthermore, if \(M\) is orientable, we write \(e(M) \in H^m(M;\Z)\) for the Euler class of the tangent bundle \(TM\) of \(M\). \par 

We denote by \(\sigma_r(x_1,\dots,x_n)\) the \(r\)--th elementary symmetric polynomial in the variables \(x_1,\dots, x_n\), with \(\sigma_0(x_1,\dots, x_n) = 1\) and \(\sigma_r(x_1,\dots, x_n) = 0\) if \(r < 0\) or \(r >n\). Finally, for any space \(X\), we write \(\beta \colon H^q(X;\Z_2) \to H^{q+1}(X;\Z)\) for the Bockstein homomorphism associated to the exact coefficient sequence \(\Z \to \Z \to \Z_2.\)

\begin{thm}[Massey--Szczarba]\label{thm:mands}
Let \(M\) be a smooth closed connected \(m\)--manifold. If \({\rm{pspan}}(M) \geq k\), then there exist elements
\begin{itemize}
    \item \(x_1,\dots, x_k \in H^1(M;\Z_2)\);
    \item \(u_j \in H^j(M;\Z_2)\), for \(1 \leq j \leq m-k\); and 
    \item \(v_j \in H^{4j}(M;\Z)\), for \(1 \leq j \leq m - k/4\)
\end{itemize} 
satisfying
\[w_q(M) = \sum_{i+j=q}u_j\sigma_i(x_1,\dots,x_k),\] for \(1 \leq q \leq m\), and  
    \begin{multline*}
        p_q(M) = \sum_{i+j=q} \left\{v_i\sigma_{2j}(\beta x_1, \dots, \beta x_k) + \left[(\beta u_{2i})^2 + \beta u_1v_i\right]\sigma_{2j-1}(\beta x_1, \dots, \beta x_k) \right\},
    \end{multline*}
for \(1 \leq q \leq m/4\).

Furthermore, if \(M\) is orientable, then \(e(M) = 0.\)
\end{thm}

The conditions of Theorem \ref{thm:mands}, give necessary and sufficient conditions for a \(4\)--manifold \(M\) to be line element parallelizable \cite[Theorem B]{MS}. Let us observe that these conditions are not sufficient for \({\rm{dim}}(M) >4.\) \par

\begin{ex}
Take any non-parallelizable odd dimensional sphere \(S^{2d+1}\), for \(d\geq 2\). The characteristic classes of such a sphere vanish, thereby trivially satisfying the conditions of Theorem \ref{thm:mands} for $k=2d+1$. However, \({\rm{pspan}}(S^{2d+1}) = {\rm{span}}(S^{2d+1})\), so \(S^{2d+1}\) is not line element parallelizable. 
\end{ex}

\begin{ex}\label{ex:pspanWall}
We illustrate how Theorem \ref{thm:mands} can be used to bound projective span from above, by showing that 
$\pspan(Q(m,n))\le m+1$ whenever $m,n>0$ with $n$ even. 

Suppose that $\pspan(Q(m,n))>m+1$. Then by Theorem \ref{thm:mands} there are classes $x_1,\ldots , x_{m+2}\in H^1(Q(m,n);\Z_2)$ and $u_j\in H^j(Q(m,n);\Z_2)$ for $j=1,\ldots, 2n-1$ such that 
\[
\begin{aligned}
w_{2n}(Q(m,n)) & = \sum_{i+j=2n} u_j \sigma_i(x_1,\ldots , x_{m+2}) \\
& = \sigma_{2n}(x_1,\ldots , x_{m+2}) +\cdots + u_{2n-1}\sigma_1(x_1,\ldots , x_{m+2}).
\end{aligned}
\]
Note that each term in this expression is divisible by a class in $H^1(Q(m,n);\Z_2)$, therefore $w_{2n}(Q(m,n))$ is in the ideal generated by $H^1(Q(m,n);\Z_2)$.

On the other hand there is a fibre bundle
\[
\xymatrix{
\C P^n \ar[r]^i & Q(m,n) \ar[r] & Q(m,0)
}
\]
induced from the projection $P(m,n)\to \R P^m$ on taking mapping tori, which gives a splitting of the tangent bundle as $TQ(m,n)\cong T_v\C P^n \oplus TQ(m,0)$. Here $T_v\C P^n$ is the vertical tangent bundle, which restricted to the fibre is $T\C P^n$. Therefore $i^*TQ(m,n)\cong T\C P^n \oplus \varepsilon^{m+1}$, from which we see that 
\[
i^*w_{2n}(Q(m,n))=w_{2n}(\C P^n).
\]
 
 Recall that $H^*(\C P^n;\Z_2)\cong \Z_2[a]/(a^{n+1})$ where $|a|=2$. For $n$ even $\chi(\C P^n)$ is odd, and since the top Stiefel--Whitney class is the mod $2$ reduction of the Euler class we see that $w_{2n}(\C P^n)=a^n\neq 0$. This gives a contradiction, since $i^*$ is a graded ring homomorphism which must therefore annihilate the ideal generated by $H^1(Q(m,n);\Z_2)$.
\end{ex}

\begin{rem}
This example will be subsumed by Theorem \ref{thm:main}, in which we calculate the projective span of all Wall manifolds. The upper bound in Section \ref{sec:upper} uses Steer's calculation \cite{Steer} of the stable span of $\C P^n$, which employs \(K\)--theoretic methods. It seems unlikely that ordinary cohomology can give sharp upper bounds in the case of $n$ odd. 
\end{rem}

\section{Upper Bounds}\label{sec:upper}
In this section, we begin the proof of Theorem \ref{thm:main} by proving the upper bound. 
\begin{prop}\label{prop:upper}
Given \(m \geq 1\) and \(n \geq 0\), the projective span of the Wall manifolds satisfies \(\pspan(Q(m,n)) \leq 
2\nu(n+1) + m + 1\), where \(\nu(k)\) denotes the largest exponent such that the corresponding \(2\)--power divides the integer \(k\). 
\end{prop}
To prove Proposition \ref{prop:upper}, we introduce the notion of stable projective span. First recall the notion of stable span. For more information about stable span in particular, we recommend the survey of Korba\v{s} and Zvengrowski \cite[Section 2]{KZ}.

\begin{defn}\label{defn:sspan}
Let \(M\) be a smooth manifold. The \emph{stable span} of \(M\), denoted \(\sspan(M)\), is defined by 
\[
\sspan(M) = \max_{r \geq 1}\left\{\span(TM \oplus \varepsilon^r) -r\right\}.
\]
\end{defn}

Here and below, the span $\span(\xi)$ and projective span $\pspan(\xi)$ of a general vector bundle $\xi$ are defined in the obvious way, in terms of splittings into direct sums. Clearly, for any manifold \(M\), we always have 
\[
{\rm{dim}}(M) \geq \sspan(M) \geq \span(M),
\]
but span and stable span can differ significantly. For example, any sphere \(S^k\) is stably parallelizable, but \(\span(S^k) = 0\) for \(k\) even. For situations where span and stable span coincide, see \cite[\S 20]{Koschorke}. 

The proceeding example plays a critical role in the determination of the projective span of the Wall manifolds.

\begin{ex}[Stable span of complex projective space]
Steer \cite{Steer} has calculated that
\begin{equation}\label{eq:Steer}
    \sspan(\C P^n) = 2\nu(n+1).
\end{equation}

The proof in \cite{Steer} uses Clifford modules to obtain the lower bound, and can be adapted also to real and quaternionic projective spaces. Here, since we are interested mainly in the complex case and require our vector fields to behave a certain way with respect to complex conjugation, we give a somewhat more explicit treatment in Sections \ref{sec:lower} and \ref{sec:Aj}.

Merely for the purposes of illustration, we enumerate a portion of the sequence of the \(2\)--adic valuation of \(n+1\) in Table 1 below along with the corresponding values for the stable span of \(\C P^n.\) Noting that \(\nu(n+1) = 0\) for \(n\) even, we have omitted all zeros from the table. 
\begin{table}[h!]
\label{table:nu}
\caption{}
\resizebox{\textwidth}{!}{%
\begin{tabular}{@{}ccccccccccccccccc@{}}
\toprule
\(n+1\)            & 2 & 4 & 6 & 8 & 10 & 12 & 14 & \(\cdots\) & 28 & 30 & 32 & 34 & 36 & 38 & \(\cdots\) & 1024 \\ \midrule
\(\nu(n+1)\)       & 1 & 2 & 1 & 3 & 1  & 2  & 1  & \(\cdots\) & 2  & 1  & 5  & 1  & 2  & 1  & \(\cdots\) & 10   \\
\(\sspan(\C P^n)\) & 2 & 4 & 2 & 6 & 2  & 4  & 2  & \(\cdots\) & 4  & 2  & 10 & 2  & 4  & 2  & \(\cdots\) & 20   \\ \bottomrule
\end{tabular}%
}
\end{table}
\end{ex}

\begin{defn}\label{defn:spspan}
Let \(\xi\) be a vector bundle. The \emph{stable projective span of \(\xi\)}, denoted \(\spspan(\xi)\), is defined by 
\[
\spspan(\xi) = \max_{r \geq 1}\left\{\pspan(\xi \oplus \varepsilon^r) - r\right\}.
\]
\end{defn}

\begin{ex}[Projective Stiefel Manifolds]
The \emph{projective Stiefel manifold} \(X_{n,k}\) is the quotient of the usual Stiefel manifold
\(
V_{n,k}/\sim
\)
where the equivalence relation is given by identifying each \(k\)--frame \((v_1,\dots, v_k)\) with the \(k\)--frame \((-v_1, \dots, -v_k).\) Despite the fact that the ordinary Stiefel manifolds are parallelizable, almost all of the projective Stiefel manifolds are not even stably parallelizable, see \cite[Theorem 3.2.5]{KZ}.\footnotemark

Interestingly, Lam \cite{Lam} gives the following description of the stable tangent bundle of \(X_{n,k}\).
\[
TX_{n,k} \oplus 
\begin{pmatrix}
k + 1 \\
2
\end{pmatrix} \varepsilon^1 = nk\xi_{n,k},
\]
where \(\xi_{n,k}\) is the line bundle associated with the double covering \(V_{n,k} \to X_{n,k}.\) Hence we see that all \(X_{n,k}\) are stably line element parallelizable.

\footnotetext{Computing the span of \(X_{n,k}\) is classically important due to connections with immersions of projective space \cite[pp.\ 21--23]{KZ}. Additionally, knowledge of \(\span(X_{n,k})\) has applications to the generalized vector field problem---i.e.\ the determination of the span of multiples of the Hopf line bundle \(\xi_{n,1}\) over \(\R P^{n-1}.\)}
\end{ex}

We can now state the following useful lemma. 
\begin{lem}\label{lem:spspanfiber}
Let \(E \to B\) be a smooth fibre bundle with compact connected fibre \(F\). Then 
\[
\pspan(E) \leq \spspan(F) + {\rm{dim}}(B).
\]
\end{lem}
\begin{proof}
The proof mimics the usual proof for span. Letting \(i \colon F \xhookrightarrow{} E\) denote the fibre inclusion, there is a bundle isomorphism \(i^*TE \cong TF \oplus \varepsilon^{{\rm{dim}}(B)}.\) Therefore, if \(TE\) splits off \(k\) line bundles, then so does \(TF \oplus \varepsilon^{{\rm{dim}}(B)},\) implying that \(\spspan(F) + {\rm{dim}}(B) \geq \pspan(E),\) as desired.
\end{proof}

\begin{proof}[Proof of Proposition \ref{prop:upper}]
We apply Lemma \ref{lem:spspanfiber} to the fiber bundle
\[
\C P^n \longrightarrow Q(m,n) \longrightarrow Q(m,0).
\]
This gives 
\begin{equation*}
\begin{aligned} 
    \pspan(Q(m,n)) & \leq \spspan(\C P^n) + {\rm{dim}}(Q(m,0)) \\
    & = \spspan(\C P^n) + m + 1 \\
    & = \sspan(\C P^n) + m + 1 \\
    & = 2\nu(n+1) + m + 1.
\end{aligned}
\end{equation*}
Here the penultimate equality holds by Proposition \ref{prop:props}(iv) since \(\C P^n\) is simply-connected, and the last equality holds by (\ref{eq:Steer}).
\end{proof}

\section{Lower Bounds}\label{sec:lower}
In this section we fix $m\ge1$ and $n\ge0$, and explain how to find optimal lower bounds for the projective span of the Wall manifold $Q(m,n)$, thereby completing the proof of Theorem \ref{thm:main}. Henceforth, write \(\nu \coloneqq \nu(n+1)\) and \(\delta \coloneqq 2\nu + m + 1.\)

\begin{prop}\label{prop:lower}
Let \(m \geq 1\) and \(n \geq 0.\) Then \(
\pspan(Q(m,n)) \geq \delta
\).
\end{prop}

The proof is constructive. We will write down $\delta$ linearly independent line fields on \(Q(m,n)\) which descend from linearly independent vector fields on $\C P^n\times S^m\times S^1$ with suitable invariance properties. The method is encapsulated by the following definition. 

\begin{defn}
Let $M$ be a smooth $G$-manifold. A vector field $\xi: M\to TM$ is called \emph{$G$-quasi-invariant} if for all $g\in G$ there exists $\varepsilon_g\in \{\pm 1\}$ such that 
\[
dg\big(\xi(x)\big) = \varepsilon_g \xi(gx)\qquad\mbox{for all }x\in M.
\]
\end{defn}

If $G$ is a finite group acting freely on $M$, then the quotient manifold $M/G$ has tangent bundle isomorphic to the quotient bundle $(TM)/G$. It is easy to see that linearly independent $G$-quasi-invariant vector fields on $M$ yield linearly independent line fields on $M/G$.

We begin by fixing some notation. Endow $\R^{m+1}$ with the usual Euclidean inner product $\la\cdot,\cdot\ra_\R$, and regard $S^m\subseteq\R^{m+1}$ as the set of vectors $v$ with $\langle v,v\rangle_\R =1$. We identify the total space $TS^m$ of the tangent bundle with 
\[
\{(v,u)\in S^m\times \R^{m+1} \mid \langle v,u\rangle_\R = 0\}.
\]

Regard $S^1\subseteq \C$ as the set of complex numbers of unit modulus so that the total space $TS^1$ can be identified with 
\[
\{(\lambda, \mu)\in S^1\times \C \mid \lambda\overline{\mu}\in i\R\}.
\]

Finally, equip $\C^{n+1}$ with the standard Hermitian inner product $\la\cdot, \cdot\ra$, and regard $S^{2n+1}\subseteq\C^{n+1}$ as the set of complex vectors $\mathbf{z}$ with $\langle \mathbf{z},\mathbf{z}\rangle=1$. Complex projective space $\C P^n$ is subsequently the quotient $S^{2n+1}/S^1$. To describe the total space $T\C P^n$ we define
\[
W_{n+1,2}=\{(\mathbf{z},\mathbf{w})\in S^{2n+1}\times\C^n \mid \la \mathbf{z},\mathbf{w}\ra=0\}.
\]
The group $S^1\subseteq \C$ acts diagonally on $W_{n+1,2}$. Identify the quotient space $W_{n+1,2}/S^1$ with $T\C P^n$ and the map induced by projection onto the first coordinate with the bundle projection $T\C P^n\to \C P^n$.

There are smooth involutions $\sigma$ and $\tau$ on $\C P^n \times S^m\times S^1$ given by
\[
\sigma\big([\mathbf{z}],v,\lambda\big)=    \big([\mathbf{\overline{z}}],-v,\lambda\big), \qquad \tau\big([\mathbf{z}],v,\lambda\big) = \big([\mathbf{z}],\rho(v),-\lambda\big).
\]
Here $\overline{\mathbf{z}}\in \C^{n+1}$ is the vector obtained from $\mathbf{z}\in \C^{n+1}$ by complex conjugation in each component, and $\rho:\R^{m+1}\to \R^{m+1}$ denotes reflection in the plane $x_{m+1}=0$. It is clear that $\sigma$ and $\tau$ commute with one another, and therefore generate a free action of $\Z_2\times\Z_2$ on $\C P^n\times S^m\times S^1$. Ultimately, the quotient \((\C P^n \times S^m \times S^1)/\Z_2 \times \Z_2\) is the Wall manifold $Q(m,n)$.

The tangent bundle $TQ(m,n)$ is therefore the quotient bundle $(T\C P^n\times TS^m\times TS^1)/\Z_2\times\Z_2$. Under the above identifications, the differentials of the involutions $\sigma$ and $\tau$ act on $T\C P^n\times TS^m\times TS^1$ as follows:
\begin{equation}\label{eq:sigma}
d\sigma\big([\mathbf{z},\mathbf{w}],(v,u),(\lambda,\mu)\big) = \big([\overline{\mathbf{z}},\overline{\mathbf{w}}],(-v,-u),(\lambda,\mu)\big),
\end{equation}
\begin{equation}\label{eq:tau}
d\tau \big([\mathbf{z},\mathbf{w}],(v,u),(\lambda,\mu)\big) =\big([\mathbf{z},\mathbf{w}],(\rho(v),\rho(u)),(-\lambda,-\mu)\big).
\end{equation}

Thus to prove Proposition \ref{prop:lower}, it suffices by the above description of $TQ(m,n)$ to find $\delta$ linearly independent vector fields $\xi_1,\ldots , \xi_\delta$ on the manifold $\C P^n\times S^m\times S^1$ which satisfy 
\begin{equation}\label{eq:invariance}
d\sigma\circ \xi_j = \pm \xi_j \circ\sigma\; \text{ and }\; d\tau \circ\xi_j = \pm \xi_j\circ\tau, \text{ for }j=1,\ldots,\delta.
\end{equation}
For properties (\ref{eq:invariance}) say precisely that the $\xi_j$ are $\Z_2\times\Z_2$--quasi-invariant and consequently descend to $\delta$ well-defined and linearly independent line fields $\{\xi_j\}$ on the quotient manifold $Q(m,n)=(\C P^n\times S^m\times S^1)/\mathbb{Z}_2\times\Z_2$.

\subsection{Constructing the vector fields: Part I}
The last $m$ vector fields arise from the trivialisation of $TS^m\times TS^1$ viewed as the stable tangent bundle of $S^m$. Let $e_1,\ldots, e_{m+1}$ be the standard basis vectors for $\R^{m+1}$. For $j=2,\ldots , m+1$, set
\begin{equation}\label{eq:lastmvecs}
\xi_{2\nu+j} \big([\mathbf{z}],v,\lambda\big):=\big([\mathbf{z},\mathbf{0}],(v,e_j-\la v,e_j\ra_\R v),(\lambda, -i\la v,e_j\ra_\R\lambda)\big).
\end{equation}
 It is easily checked using formulae (\ref{eq:sigma}) and (\ref{eq:tau}) that 
 \begin{enumerate}[label=\(\bullet\)]
     \item \(d\sigma \circ\xi_{2\nu+j} = - \xi_{2\nu+j} \circ\sigma\), for \(2\le j \le m+1\),
    \item \(d\tau \circ\xi_{2\nu+j} =\xi_{2\nu+j} \circ\tau\), for \(2\le j \le m\), and
 \item \(d\tau\circ\xi_{2\nu + m + 1} = -\xi_{2\nu + m + 1}\circ\tau.\)
 \end{enumerate}
 (Here, we use that $\rho$ is an orthogonal involution fixing $e_2,\ldots,e_m$ and negating $e_{m+1}$, so that $\la \rho(v), e_j\ra_\R=\la v,e_j\ra_\R$ for $j=2,\ldots , m$ and $\la \rho(v), e_{m+1} \ra_\R=-\la v,e_{m+1}\ra_\R$.)  Therefore, the \(\xi_{2\nu+j}\) are \(\Z_2 \times \Z_2\)--quasi-invariant and thus descend to well-defined line fields \(\{\xi_{2\nu+j}\}\) on \(Q(m,n)\) for \(j = 2,\dots, m+1.\)
 
 \subsection{Constructing the vector fields: Part II}
  
In order to construct the first $2\nu+1$ vector fields, we can use the trivialization of \(TS^m \times TS^1 \cong \varepsilon^{m+1}\) to construct more vector fields on \(\C P^n \times S^m \times S^1\), since \(\sspan(\C P^n) > 0\) for \(n\) odd by by (\ref{eq:Steer}).

To this end, we need the following result, the proof of which will be postponed until Section \ref{sec:Aj}. First, recall that  complex conjugation on \(\C^{n+1}\) is anti-unitary with respect to the standard Hermitian inner product on \(\C^{n+1}\). That is 
\begin{equation}\label{eq:antiunitary}
    \langle \overline{\mathbf{z}}, \overline{\mathbf{w}} \rangle = \overline{\langle \mathbf{z}, \mathbf{w}\rangle},
\end{equation}
for all \(\mathbf{z}, \mathbf{w} \in \C^{n+1}.\)

\begin{lem}\label{lem:Aj}
For \(j = 1, \dots, 2\nu +1\), there exist automorphisms \(A_j \colon \C^{n+1} \to \C^{n+1}\) with the following properties:
\begin{equation}\label{eq:anticomm}
A_j\circ A_k = -A_k\circ A_j\quad\mbox{ for all }j\neq k,
\end{equation}
\begin{equation}\label{eq:skewhermitian}
    \langle \mathbf{z}, A_j(\mathbf{w})\rangle + \langle A_j(\mathbf{z}), \mathbf{w} \rangle = 0\quad\mbox{ for all }\mathbf{z}, \mathbf{w} \in \C^{n+1},
\end{equation}
\begin{equation}\label{eq:commconj}
    \text{There exists } \varepsilon_j \in \{\pm 1\} \text{ such that } A_j(\overline{\mathbf{z}}) =\varepsilon_j\overline{A_j(\mathbf{z})}\quad\mbox{ for all }\mathbf{z} \in \C^{n+1}.
\end{equation}
\end{lem}

Assuming Lemma \ref{lem:Aj}, we can define the remaining vector fields. To ease the notation, we set 
\[
\beta_j(\mathbf{z}):=\la \mathbf{z}, A_j(\mathbf{z})\ra.
\]
 Note that
\begin{enumerate}[label=\(\bullet\)]
    \item Property (\ref{eq:skewhermitian}) with $\mathbf{w}=\mathbf{z}$ together with the usual conjugate symmetry implies that $\beta_j(\mathbf{z})$ is purely imaginary; and 
    \item Properties (\ref{eq:antiunitary}) and (\ref{eq:commconj}) together imply that  $\beta_j(\overline{\mathbf{z}})=\varepsilon_j \overline{\beta_j(\mathbf{z})}=-\varepsilon_j \beta_j(\mathbf{z})$ for some $\varepsilon_j\in \{\pm 1\}$. 
\end{enumerate} Subsequently, for $j=1,\ldots , 2\nu+1$, we define
\begin{equation}\label{eq:first2v+1vecs}
\begin{aligned}
 \xi_{j}\big([\mathbf{z}],v,\lambda\big):= &\big([\mathbf{z}, A_j(\mathbf{z}) + \beta_j(\mathbf{z}) \mathbf{z}] , \\
 &\big(v, i\beta_j(\mathbf{z}) (e_1- \la v,e_1\ra_\R v)\big), \\
 & (\lambda, \beta_j(\mathbf{z})\la v,e_1\ra_\R \lambda )\big). 
 \end{aligned}
\end{equation}

We now check that the vector fields defined in (\ref{eq:first2v+1vecs}) satisfy the desired properties. 
\begin{lem}\label{lem:invariance}
For \(j = 1, \dots, 2\nu+1\), the vector fields \(\xi_j\) of \emph{(\ref{eq:first2v+1vecs})} are well-defined and satisfy the invariance properties \emph{(\ref{eq:invariance})}.
\end{lem}
\begin{proof}
Firstly, \(\xi_j\) is well-defined since
\begin{enumerate}[label=\(\bullet\)]
    \item \(
\la \mathbf{z}, A_j(\mathbf{z}) + \beta_j(\mathbf{z}) \mathbf{z} \ra = \beta_j(\mathbf{z}) + \overline{\beta_j(\mathbf{z})}\la \mathbf{z},\mathbf{z}\ra =\beta_j(\mathbf{z}) -\beta_j(\mathbf{z}) \cdot 1 = 0,
\) and also
    \item $\beta_j(\omega\mathbf{z})=\beta_j(\mathbf{z})$ implies $\xi_{j}([\omega\mathbf{z}],v,\lambda)=\xi_{j}([\mathbf{z}],v,\lambda)$ for all $\omega\in S^1\subseteq \C$. 
\end{enumerate}
We now check the invariance property under $\sigma$. Using formula (\ref{eq:sigma}) we have
\begin{multline*}
d\sigma\circ \xi_{j} \big([\mathbf{z}],v,\lambda\big) \\
 = \big([\overline{\mathbf{z}}, \overline{A_j(\mathbf{z}) + \beta_j(\mathbf{z}) \mathbf{z}}] , (-v, -i\beta_j(\mathbf{z}) e_1+i\beta_j(\mathbf{z})\la v, e_1\ra_\R v)), (\lambda, \beta_j(\mathbf{z})\la v,e_1\ra_\R \lambda)\big)\\
\; =  \big([\overline{\mathbf{z}}, \overline{A_j(\mathbf{z})} - \beta_j(\mathbf{z}) \overline{\mathbf{z}}] , (-v, -i\beta_j(\mathbf{z}) e_1+i\beta_j(\mathbf{z})\la v, e_1\ra_\R v)), (\lambda, \beta_j(\mathbf{z})\la v,e_1\ra_\R \lambda)\big).
\end{multline*}
On the other hand, we have
\begin{multline*}
\xi_{j}\circ \sigma \big([\mathbf{z}],v,\lambda\big) 
 = \xi_{j}  \big([\overline{\mathbf{z}}],-v,\lambda\big) \\
                   \hspace{-0.8cm}     = \big([\overline{\mathbf{z}}, A_j(\overline{\mathbf{z}}) + \beta_j(\overline{\mathbf{z}})\overline{\mathbf{z}}] , (-v, i\beta_j(\overline{\mathbf{z}}) e_1-i\beta_j(\overline{\mathbf{z}})\la v, e_1\ra_\R v)), (\lambda, -\beta_j(\overline{\mathbf{z}})\la v, e_1\ra_\R \lambda)\big) \\
                           = \big([\overline{\mathbf{z}}, \varepsilon_j\overline{A_j(\mathbf{z})} -\varepsilon_j \beta_j(\mathbf{z})\overline{\mathbf{z}}] , (-v, -\varepsilon_j i\beta_j(\mathbf{z}) e_1 + \varepsilon_ji\beta_j(\mathbf{z})\la v, e_1\ra_\R v)), (\lambda, \varepsilon_j\beta_j(\mathbf{z})\la v, e_1\ra_\R \lambda)\big) \\
                           = \varepsilon_j d\sigma\circ \xi_{j} \big([\mathbf{z}],v,\lambda\big).\hspace{10cm}
\end{multline*}
Now we check invariance under $\tau$. Using formula (\ref{eq:tau}) we have
\begin{multline*}
d\tau\circ \xi_{j} \big([\mathbf{z}],v,\lambda\big) \\
\hspace{-0.6cm}  = \big([\mathbf{z}, A_j(\mathbf{z}) + \beta_j(\mathbf{z}) \mathbf{z}] , \big(\rho(v), i\beta_j(\mathbf{z}) (e_1-\la v, e_1\ra_\R \rho(v))\big), (-\lambda, -\beta_j(\mathbf{z})\la v, e_1\ra_\R \lambda)\big)\\
     = \big([\mathbf{z}, A_j(\mathbf{z}) + \beta_j(\mathbf{z}) \mathbf{z}] , \big(\rho(v), i\beta_j(\mathbf{z}) (e_1-\la \rho(v), e_1\ra_\R \rho(v))\big), (-\lambda, -\beta_j(\mathbf{z})\la \rho(v), e_1\ra_\R \lambda)\big) \\
     = \xi_{j}\circ \tau \big([\mathbf{z}],v,\lambda\big).\hspace{10cm}
\end{multline*}
(Since $\rho$ is an orthogonal involution fixing $e_1$ we have $\la \rho(v),  e_1\ra_\R=\la v, e_1\ra_\R$.)
\end{proof}

\subsection{Constructing the vector fields: Part III}
In order to complete the proof of Proposition \ref{prop:lower}, it remains only to show that the vector fields $\xi_1,\ldots , \xi_{\delta}$ defined above in (\ref{eq:lastmvecs}) and (\ref{eq:first2v+1vecs}) are linearly independent. This is equivalent to the following lemma. 
\begin{lem}\label{lem:linearindep} 
All of the following are satisfied. 
\begin{enumerate}[label=\emph{(\roman*)}]
    \item The first \(2\nu +1\) vector fields \(\{\xi_1, \dots, \xi_{2\nu+1}\}\) as defined in \emph{(\ref{eq:first2v+1vecs})} form a linearly independent set. 
    \item The last \(m\) vector fields \(\{\xi_{2\nu+2}, \dots, \xi_{2\nu + m+1}\}\) as defined in \emph{(\ref{eq:lastmvecs})} form a linearly independent set.
    \item \(
    {\rm{Span}}(\xi_1, \dots, \xi_{2\nu+1} ) \cap {\rm{Span}}(\xi_{2\nu+2}, \dots, \xi_{2\nu + m+1})  = \{0\}
    \). 
\end{enumerate}
\end{lem}
\begin{proof}
\noindent{(i)}~~Given $([z],v,\lambda)\in \C P^n\times S^m\times S^1$, suppose there are $x_1,\ldots ,x_{2\nu+1}\in \R$ such that $\sum_{j=1}^{2\nu+1} x_j \xi_j([z],v,\lambda)=0$. We obtain equations
\begin{equation}\label{eq:z}
\sum_{j=1}^{2\nu+1} x_j(A_j(\mathbf{z}) + \beta_j(\mathbf{z})\mathbf{z}) = 0,
\end{equation}
\begin{equation}\label{eq:v}
\sum_{j=1}^{2\nu+1} ix_j\beta_j(\mathbf{z})(e_1-\la v,e_1\ra_\R v) = 0,
\end{equation}
\begin{equation}\label{eq:lambda}
\sum_{j=1}^{2\nu+1} x_j\beta_j(\mathbf{z})\la v,e_1\ra_\R \lambda = 0.
\end{equation}

As $\lambda$ is a nonzero complex number, Equation (\ref{eq:lambda}) implies $\sum x_j \beta_j(\mathbf{z})\la v,e_1\ra_\R = 0$. It follows from Equation (\ref{eq:v}) that $\sum i x_j\beta_j(\mathbf{z})e_1=0$, and hence that $\sum x_j\beta_j(\mathbf{z})=0$. In turn, from Equation (\ref{eq:z}) we obtain $\sum x_jA_j(\mathbf{z})=0$. 

We now demonstrate that the vectors $A_j(\mathbf{z})\in \C^{n+1}\cong\R^{2n+2}$ are mutually orthogonal, hence linearly independent, over $\R$. It suffices to show that $\la A_j(\mathbf{z}),A_k(\mathbf{z})\ra$ is purely imaginary for $j\neq k$, and this follows from properties (\ref{eq:anticomm}) and (\ref{eq:skewhermitian}) of Lemma \ref{lem:Aj} together with conjugate symmetry:
\begin{equation*}
\begin{aligned}
\la A_j(\mathbf{z}),A_k(\mathbf{z})\ra & = -\la \mathbf{z}, A_jA_k(\mathbf{z}) \ra \\
                                                          & = \la \mathbf{z}, A_kA_j(\mathbf{z}) \ra \\
                                                          & = -\la A_k(\mathbf{z}),A_j(\mathbf{z}) \ra \\
                                                          & = - \overline{\la A_j(\mathbf{z}),A_k(\mathbf{z}) \ra}.
\end{aligned} 
\end{equation*}

Finally, the equation $\sum x_jA_j(\mathbf{z})=0$ implies $x_1=\cdots = x_{2\nu+1}=0$ as required.

\noindent{(ii)}~~Given $([z],v,\lambda)\in \C P^n\times S^m\times S^1$, suppose there are $x_2,\ldots ,x_{m+1}\in \R$ such that $\sum_{j=2}^{m+1} x_j \xi_{2\nu+j}([z],v,\lambda)=0$. We obtain equations
\[
\sum_{j=2}^{m+1} x_j(e_j-\la v,e_j\ra_\R v) =0,\qquad \sum_{j=2}^{m+1} -i x_j \la v,e_j\ra_\R \lambda =0,
\]
which quickly imply, together with the linear independence of $\{e_2,\ldots, e_{m+1}\}$, that $x_2=\cdots = x_{m+1}=0$.

\noindent{(iii)}~~Given $([z],v,\lambda)\in \C P^n\times S^m\times S^1$, suppose there are real scalars $x_1,\ldots , x_{2\nu+1}, \\ y_2,\ldots , y_{m+1}\in\R$ such that $\sum_{j=1}^{2\nu+1} x_j \xi_j([z],v,\lambda) = \sum_{k=2}^{m+1} y_k\xi_{2\nu+k}([z],v,\lambda)$. From the $TS^1$ component we get
\[
\sum_{j=1}^{2\nu+1} x_j\beta_j(\mathbf{z})\la v,e_1\ra_\R \lambda = \sum_{k=2}^{m+1} -i y_k \la v,e_j\ra_\R \lambda
\]
from which we deduce that 
\begin{equation}\label{eq:TS1}
\sum_{j=1}^{2\nu+1}i x_j \beta_j(\mathbf{z})\la v,e_1\ra_\R  = \sum_{k=2}^{m+1}  y_k \la v,e_j\ra_\R. 
\end{equation}
Now from the $TS^m$ component we have
\[
\sum_{j=1}^{2\nu+1} i x_j \beta_j(\mathbf{z}) (e_1 - \la v,e_1\ra_\R v) = \sum_{k=2}^{m+1} y_k(e_k-\la v, e_k\ra_\R v),
\]
which together with Equation (\ref{eq:TS1}) implies that 
\[
\sum_{j=1}^{2\nu+1} i x_j \beta_j(\mathbf{z}) e_1  = \sum_{k=2}^{m+1} y_k e_k.
\]
Since $\{e_1,\ldots , e_{m+1}\}$ is linearly independent, $y_2=\cdots = y_{m+1}=0$.
\end{proof}

\section{Proof of Lemma \ref{lem:Aj}}\label{sec:Aj}
In this section, we give the proof of Lemma \ref{lem:Aj} by means of Clifford modules. We refer to the book of Friedrich \cite{Friedrich} for useful background on Clifford algebras and their representations, and we adopt the notation used there. 

Let \((V, \langle \cdot, \cdot \rangle_{V})\) and \((W, \langle \cdot, \cdot \rangle_{W})\) be inner product spaces. With this data one can form an inner product on the direct sum \(V \oplus W\) by setting
\[
\langle v_1 + w_1, v_2 + w_2\rangle_{V \oplus W} = \langle v_1,v_2 \rangle_{V} + \langle w_1, w_2 \rangle_{W}.
\]
Similarly, there is an inner product on \(V \otimes W\) defined by
\[
\langle v_1 \otimes w_1, v_2 \otimes w_2 \rangle_{V \otimes W} = \langle v_1, v_2\rangle_V \cdot \langle w_1, w_2 \rangle_W. 
\]
Furthermore, recall the following definition, c.f.\ \cite[Chapter 2]{Simon}.
\begin{defn}
Let \(V\) be a complex vector space. A \emph{real structure} is an \(\R\)--linear map \(J \colon V \to V\) with the properties
\begin{enumerate}[label=(\roman*)]
    \item \(J\) is \(\C\)--antilinear; that is \(J(\alpha x + \beta y) = \overline{\alpha}J(x) + \overline{\beta}J(y)\) for \(\alpha,\beta \in \C\), \(x,y \in V.\)
    \item \(J^2 = {\rm{id}}.\)
\end{enumerate}
Moreover, we say that a real structure $J$ on a complex inner product space \((V, \langle \cdot, \cdot \rangle)\) is \emph{orthogonal} if and only if \(\langle Jx, Jy \rangle = \overline{\langle x,y\rangle}\) for all \(x,y \in V\). 
\end{defn}

Now for fixed \(n \geq 0\), we seek \(2\nu(n+1) + 1\) automorphisms \(A_j \colon \C^{n+1} \to \C^{n+1}\) satisfying properties (\ref{eq:anticomm}), (\ref{eq:skewhermitian}) and (\ref{eq:commconj}) of Lemma \ref{lem:Aj}. For $n=0$ we may take $A_1(z)=iz$. Hence we assume $n\ge1$ in what follows.

Equip \(\C^{n+1}\) with the orthogonal real structure \(J \colon \C^{n+1} \to \C^{n+1}\) defined by 
\[
J \colon \begin{pmatrix} 
    z_1 \\
    \vdots \\
    z_{n+1}
\end{pmatrix} \mapsto 
\begin{pmatrix}
    \overline{z}_1 \\
    \vdots \\
    \overline{z}_{n+1}
\end{pmatrix}.
\]
Here \(\C^{n+1}\) is endowed with its standard Hermitian inner product \(\langle \mathbf{z}, \mathbf{w}\rangle = \mathbf{w}^*\mathbf{z}\), where \(\mathbf{w}^* = (J(\mathbf{w}))^T\) denotes the conjugate transpose of \(\mathbf{w}\) with respect to the orthogonal real structure \(J\).

Write \(n + 1 = 2^\nu b\) where \(\nu = \nu(n+1)\) and \(b\) is odd. This gives rise to an identification
\begin{equation}\label{eq:isometry}
\C^{n+1} \cong \underbrace{(\C^2 \otimes \cdots \otimes \C^2)}_{\nu} \oplus \cdots \oplus \underbrace{(\C^2 \otimes \cdots \otimes \C^2)}_{\nu},
\end{equation}
where there are \(b\) direct summands and the tensor products are taken over \(\C.\) The complex vector space \(\Delta_{2\nu + 1} \coloneqq \underbrace{\C^2 \otimes \cdots \otimes \C^2}_{\nu}\) is the \emph{complex spinor space}. Denote by \(b\Delta_{2\nu+1} \coloneqq \Delta_{2\nu + 1}^{\oplus b}\) the right-hand side of (\ref{eq:isometry}). Then $b\Delta_{2\nu+1}$ admits a Hermitian inner product built from the usual Hermitian inner product on \(\C^2\) by taking tensor products and direct sums, as discussed above. 

We can define a real structure on $b\Delta_{2\nu+1}$ as follows. Endow \(\C^{2}\) with its usual complex conjugation, which we also denote by
\[
J \colon \C^2 \ni \begin{pmatrix} 
    z_1 \\
    z_{2}
\end{pmatrix} \mapsto 
\begin{pmatrix}
    \overline{z}_1 \\
    \overline{z}_{2} 
\end{pmatrix} \in \C^2. 
\]
Then set \[bJ_{2\nu+1} = \underbrace{(J \otimes \cdots \otimes J)}_{\nu} \oplus \cdots \oplus \underbrace{(J \otimes \cdots \otimes J)}_{\nu},\] where there are \(b\) direct summands. It is straightforward to check that \(bJ_{2\nu+1}\) is a well-defined orthogonal real structure on \(b\Delta_{2\nu+1}\), c.f.\ \cite[Section 1.7]{Friedrich}.

The following result follows from direct (if tedious) computation. 

\begin{lem}\label{lem:isometry}
There exists a unitary isometry \(T \colon b\Delta_{2\nu+1}\to \C^{n+1}\) which maps the real structure \(bJ_{2\nu+1}\) on \(b\Delta_{2\nu+1}\) to the real structure \(J\) on \(\C^{n+1}\). 
\end{lem}

As a consequence of Lemma \ref{lem:isometry}, in order to find automorphisms \(A_j \colon \C^{n+1}\to \C^{n+1}\) for \(j = 1, \dots, 2\nu+1,\) satisfying the properties listed in Lemma \ref{lem:Aj}, it suffices to construct automorphisms \(A_j \colon b\Delta_{2\nu+1} \to b\Delta_{2\nu+1}\) for \(j = 1, \dots, 2\nu+1,\) satisfying the analogous properties with respect to the orthogonal real structure on $ b\Delta_{2\nu+1}$. This is the statement of Lemma \ref{lem:Ajprelim} below. 

We transport the problem in this way because the complex spinor space \(\Delta_{2\nu+1}\) carries useful structure as a module of a complex Clifford algebra.
We write  \(\mathcal{C}_n\) for the standard real Clifford algebra associated to the negative definite quadratic form $q(\mathbf{x})=-x_1^2-\ldots - x_n^2$ on $\R^n$, and \(\mathcal{C}_n^c\cong \mathcal{C}_n\otimes_\R \C\) for the standard complex Clifford algebra associated to the negative definite quadratic form \(q(\mathbf{z}) = -z_1^2 - \dots -z_n^2\) on $\C^n$. According to \cite[Section 1.3]{Friedrich}, we have the following identification.
\begin{equation*}
\begin{aligned}
\mathcal{C}^c_{2\nu+1} = \mathcal{C}_{2\nu + 1} \otimes_\R \C &\overset{\Phi}\cong \underbrace{M_2(\C) \otimes \cdots \otimes M_2(\C)}_\nu \oplus \underbrace{M_2(\C) \otimes \cdots \otimes M_2(\C)}_{\nu}\\
&\cong {\rm{End}}(\Delta_{2\nu+1}) \oplus  {\rm{End}}(\Delta_{2\nu+1})
\end{aligned}
\end{equation*}
Denote by \(e_i\) for \(i = 1,\dots, 2\nu + 1\), the generators of the algebra \(\mathcal{C}^c_{2\nu+1}\). Via the inclusions \(\R^{2\nu +1} \subset \mathcal{C}_{2\nu+1} \subset \mathcal{C}^c_{2\nu+1}\), we can regard each \(e_i\) as the \(i\)--th basis vector of \(\R^{2\nu + 1}\). Next let
\[
E= 
\begin{pmatrix}
1 & 0 \\
0 & 1
\end{pmatrix}, ~~
g_1 = 
\begin{pmatrix}
i & 0 \\
0 & -i
\end{pmatrix}, ~~
g_2 = 
\begin{pmatrix}
0 & i \\
i & 0
\end{pmatrix}, ~~
T = 
\begin{pmatrix}
0 & -i \\
i & 0
\end{pmatrix}
\]
be the generating elements of the algebra \(M_2(\C)\). Then the isomorphism \(\Phi\) is given as follows. If \(1 \leq j \leq 2\nu\), the generator \(e_j\) is mapped to 
\[
e_j \mapsto \left(E \otimes \cdots \otimes E \otimes g_{\alpha(j)} \otimes \underbrace{T \otimes \cdots \otimes T}_{\lfloor \frac{j-1}{2} \rfloor},\; E \otimes \cdots \otimes E \otimes g_{\alpha(j)} \otimes \underbrace{T \otimes \cdots \otimes T}_{\lfloor \frac{j-1}{2} \rfloor}\right),
\]
where \(\alpha(j) = 1\) if \(j\) is odd and \(\alpha(j) = 2\) if \(j\) is even. On the other hand, if \(j = 2\nu + 1\), the generator \(e_{2\nu+1}\) is mapped to 
\[
e_{2\nu+1} \mapsto (iT \otimes \cdots \otimes T, -iT \otimes \cdots \otimes T).
\]
Now, the complex spinor space \(\Delta_{2\nu +1}\) becomes a module over the complex Clifford algebra \(\mathcal{C}^c_{2\nu+1}\) via the \emph{spin representation}
\[
\kappa \colon \mathcal{C}^c_{2\nu+1} \overset{\Phi}\cong {\rm{End}}(\Delta_{2\nu+1}) \oplus  {\rm{End}}(\Delta_{2\nu+1}) \overset{{\rm{proj}}_1}\longrightarrow {\rm{End}}(\Delta_{2\nu+1}),
\]
whence a vector \(x \in \R^{2\nu+1} \subset \mathcal{C}_{2\nu+1} \subset \mathcal{C}_{2\nu+1}^c \overset{\kappa}\to {\rm{End}}(\Delta_{2\nu+1})\) can be viewed as an endomorphism of \(\Delta_{2\nu+1}\). This leads to the so-called \emph{Clifford multiplication of vectors and spinors}, which is a linear map 
\begin{equation*}
\begin{aligned}
\mu \colon \R^{2\nu+1} \otimes_\R \Delta_{2\nu+1} &\to \Delta_{2\nu+1}\\
(x,\varphi) & \mapsto \kappa(x)(\phi) \coloneqq x \cdot \varphi.
\end{aligned}
\end{equation*}

It remains to prove the following lemma, which, in combination with Lemma \ref{lem:isometry}, proves Lemma \ref{lem:Aj}. 

\begin{lem}\label{lem:Ajprelim}
For \(j = 1, \dots, 2\nu+1\), define $A_j:  b\Delta_{2\nu+1}\to  b\Delta_{2\nu+1}$ by
\[
A_j = \underbrace{\kappa(e_j) \oplus \cdots \oplus\kappa(e_j)}_{b}.
\]
    \begin{enumerate}[label=\emph{(\roman*)}]
    \item The $A_j$ anti-commute: for all $j\neq k$ we have $A_j\circ A_k = -A_k\circ A_j$.
        \item The \(A_j\) are skew--Hermitian. That is, for any \(\varphi, \psi \in b\Delta_{2\nu+1}\), the \(A_j\) satisfy \(\langle \varphi, A_j(\psi) \rangle + \langle A_j(\varphi), \psi\rangle = 0\).
        \item The complex conjugate \(bJ_{2\nu+1}\) on \(b\Delta_{2\nu+1}\) commutes with $A_j$ up to sign. 
    \end{enumerate}
\end{lem}

\begin{proof}[Proof of \emph{(i)}]
The generators $e_j$ of $\mathcal{C}^c_{2\nu+1}$ anti-commute, and $\kappa: \mathcal{C}^c_{2\nu+1}\to \mathrm{End}(\Delta_{2\nu+1})$ is an algebra homomorphism.
\end{proof}

\begin{proof}[Proof of \emph{(ii)}]
Recall that for matrices \(X\) and \(Y\) the conjugate transpose distributes over the direct sum \((X\oplus Y)^* = X^* \oplus Y^*\) as well as over the tensor product \((X\otimes Y)^* = X^* \otimes Y^*\)---see e.g.\ \cite[Exercise 2.28]{CN}. Thus, for each \(j = 1, \dots, 2\nu+1\), the operator \(A_j\) satisfies
\[
\langle \varphi, A_j(\psi) \rangle + \langle A_j(\varphi), \psi\rangle = 0,
\]
for any \(\varphi, \psi \in b\Delta_{2\nu+1}\) if and only if 
\[
A_j = -A_j^*,
\]
where we take the conjugate transpose of \(A_j\) with respect to the real structure \(bJ_{2\nu+1}\) on \(b\Delta_{2\nu+1}\). 
It suffices to check that the matrices
\begin{equation*}
\begin{aligned}
X(j) = \begin{cases}
E \otimes \cdots \otimes E \otimes g_{\alpha(j)} \otimes \underbrace{T \otimes \cdots \otimes T}_{\lfloor \frac{j-1}{2} \rfloor}, &(j = 1, \dots, 2\nu), \\
iT \otimes \cdots \otimes T, &(j = 2\nu + 1)
\end{cases}
\end{aligned}
\end{equation*}
are skew-Hermitian with respect to the \(J^{\otimes \nu}\) conjugation. For if \(X(j) = -X(j)^*\) then 
\[
A_j^* = \underbrace{X(j)^* \oplus \cdots \oplus X(j)^*}_b = \underbrace{-X(j) \oplus \cdots \oplus -X(j)}_b = -A_j
\]
for any \(j = 1, \dots, 2\nu +1\).

But \(E,T\) are Hermitian with respect to the complex conjugation \(J\), that is \(E = E^*\) and \(T = T^*\), while \(g_1,g_2\) are skew-Hermitian, that is \(g_i = -g_i^*\) for \(i = 1,2\). So indeed \(X(j) = -X(j)^*\) for any \(j = 1, \dots, 2\nu + 1\), as required. 
\end{proof}

\begin{proof}[Proof of \emph{(iii)}]
First note the following easy computations
\[
JE = EJ, ~~Jg_1 = -g_1J, ~~ Jg_2 = -g_2J, ~~ JT = -TJ,
\]
from which we immediately obtain the following. For \(j = 1, \dots, 2\nu\):
    \begin{equation*}
    J^{\otimes \nu}X(j)  =  
    \begin{cases}-X(j)J^{\otimes \nu}, &\text{ for } \lfloor\frac{j-1}{2}\rfloor \text{ even},\\
     X(j) J^{\otimes \nu}, &\text{ for } \lfloor\frac{j-1}{2}\rfloor \text{ odd.}
    \end{cases}
    \end{equation*}
While for \(j = 2\nu + 1\): 
    \begin{equation*}
        J^{\otimes \nu}X(j)= 
        \begin{cases}
         -X(j)J^{\otimes \nu} &\text{ for } \nu \text{ even,} \\
         X(j)J^{\otimes \nu}, &\text{ for } \nu \text{ odd.}
        \end{cases}
    \end{equation*}
Thus for any \(\varphi \in b\Delta_{2\nu+1}\) and for any \(j=1, \dots, 2\nu+1,\) 
\[
A_j(bJ_{2\nu+1}(\varphi)) = X(j)^{\oplus b}bJ_{2\nu+1}(\varphi)= \pm bJ_{2\nu +1} X(j)^{\oplus b} = \pm bJ_{2\nu+1}A_j(\varphi),
\]
as claimed.
\end{proof}

\begin{rem}
Given a complex representation $G\to \operatorname{Aut}_\C(V)$, recall that a \emph{real structure} on $V$ is an $\R$-linear and $\C$-antilinear map $J:V\to V$ such that $J^2=\operatorname{id}$, and such that $J(gv)=g J(v)$ for all $g\in G$ and $v\in V$. As noted for example in \cite[Section 1.7]{Friedrich}, the complex spinor representation $\Delta_{2\nu+1}$ does not always admit a real structure. However, what we have observed and made use of here is that it always admits a \emph{quasi-real} structure, namely a $J$ as above satisfying that for all $g\in G$ there exists a sign $\varepsilon_g\in \{\pm 1\}$ such that $J(g v) = \varepsilon_g g J(v)$ for all $v\in V$. Such quasi-real structures may lead to further computations of projective span, or be of independent interest. 
\end{rem}

\end{document}